\documentclass{article}
\usepackage{graphicx} 
\usepackage{amsthm, amsmath, amssymb}
\usepackage{todonotes,xcolor,hyperref}
\usepackage{caption,subcaption}
\usepackage[normalem]{ulem}
\newtheorem{theorem}{Theorem}[section]
\newtheorem{lemma}[theorem]{Lemma}
\newtheorem{corollary}[theorem]{Corollary}
\theoremstyle{definition}

\newtheorem{remark}[theorem]{Remark}
\newtheorem{example}[theorem]{Example}
\numberwithin{equation}{section}

\title{\sc Optimal allocation in combat models\\ with multiple battlefields}
\makeatletter
\let\@fnsymbol\@arabic
\makeatother
\author{Antonín Slavík\thanks{Charles University, Faculty of Mathematics and Physics, Czech Republic. E-mail: {\tt slavik@karlin.mff.cuni.cz}, ORCID 0000-0003-3941-7375} \mbox{ and }Vladimír Švígler\thanks{Department of Mathematics \& NTIS, Faculty of Applied Sciences, University of West Bohemia, Czech Republic. E-mail: {\tt sviglerv@kma.zcu.cz}, ORCID 0000-0003-0063-3564}}
\date{}

\begin{document}
	
	\maketitle
	
	\allowdisplaybreaks

\begin{abstract}
	We study the allocation of a force of a fixed size across independent
battlefields against a known distribution of opposing forces. Formulating battlefield outcomes through loss and survival functions allows
us to establish allocation results without assuming a specific
system of combat equations. For convex enemy loss functions,
an optimal allocation exists with at most one partially
saturated battlefield. We prove NP-hardness for every prescribed
family of strictly convex, strictly increasing loss functions
satisfying an endpoint normalization, and give an exact
pseudopolynomial algorithm for the special case when the total force size and saturation thresholds are integers. We also derive allocation rules for maximizing
successful engagements, maximizing own survivors, and maximizing
the ratio of enemy losses to own losses. Applications include a power-law family
containing the classical Lanchester, guerrilla, and mixed
combat models, as well as models with implicitly defined
outcomes.
	
\end{abstract}

\section{Introduction}

How should a force be distributed across several battlefields,
given a known distribution of opposing forces?
The answer depends both on the outcomes of individual engagements
and on the objective of the allocation.
Maximizing enemy losses, maximizing the number of successful
engagements, and minimizing own losses can lead to different
allocations of forces, even when the underlying combat model
is unchanged.

We consider two opposing armies engaged on independent
battlefields. One army's initial allocation is fixed,
while the other allocates a given number of units before
combat begins. No units are subsequently transferred
between battlefields. Rather than imposing a particular
system of combat equations throughout, we describe
battlefield outcomes through loss and survival functions.
These functions give the enemy losses and own survivors
as functions of the initial allocation. Our allocation
results depend on their structural properties and
therefore apply to different models whose outcome
functions satisfy the same assumptions.

Our principal objective is to maximize total enemy losses.
For continuous, nondecreasing loss functions that are convex
up to their saturation thresholds, an optimal allocation
exists in which all but at most one battlefield receive
either no units or exactly the number required to eliminate
the opposing force.
Despite this simple structure, finding an optimal
allocation is NP-hard for every prescribed family of
strictly convex, strictly increasing loss functions
satisfying a natural endpoint normalization.
When the total force size and saturation thresholds are
integers, we give an exact pseudopolynomial algorithm
based on dynamic programming.

We also establish allocation rules for alternative
objectives. Maximizing the number of successful
engagements requires only ordering the saturation
thresholds. Under suitable assumptions, maximizing
own survivors or the ratio of enemy losses to own
losses favors concentration on a single battlefield.
By contrast, maximizing own survivors while requiring
elimination of the opposing force on every battlefield
leads to equality of marginal survival values.
When the survival functions share a common form after
scaling by their thresholds, the optimal allocation
is proportional to those thresholds.

We apply these results to a power-law family containing
the classical Lanchester, unaimed-fire (guerrilla), and mixed
combat models. Two effective exponents determine the convexity
of the enemy loss functions and the monotonicity of
the individual loss ratios, allowing the general
allocation results to be applied uniformly across
the corresponding parameter ranges.
Further examples involve an intelligence function and
implicitly defined battlefield outcomes, illustrating
that explicit loss and survival formulas are not essential.

Section~\ref{sec:allocation-framework} introduces the
framework and illustrates it with the classical
Lanchester model.
Section~\ref{sec:max-enemy-losses} studies the structure,
complexity, and exact solution of enemy loss maximization.
Section~\ref{sec:alternative-objectives} develops allocation
rules for the alternative objectives, and
Section~\ref{sec:combat-applications} applies the results
to specific combat models.

We conclude the introduction with a brief literature survey.
Mathematical models of combat have been studied extensively
in the literature. 
Early attempts to model warfare using differential equations
date back to World War~I, when
F.~W. Lanchester~\cite{Lanchester} and M.~Osipov~\cite{Osipov}
developed models of battles between two forces engaged
in aimed fire combat.
The relation between initial and surviving force sizes
is known as the \textit{Lanchester square law}.
The square law follows from dynamics in which the attrition
rate of each force is assumed to be proportional only
to the size of the enemy force.
Other models describe unaimed fire, in which the attrition
rate is additionally proportional to the size of one's
own force, leading to the \textit{Lanchester linear law}.
Deitchman's guerrilla warfare model combines aimed fire
on one side with unaimed fire on the other~\cite{Deitchman};
we refer to this asymmetric model as the mixed combat model.
Further developments include convex interpolation between
aimed and unaimed fire~\cite{KM}.
Although these models were developed in the twentieth century,
their applications extend to much earlier forms of conflict;
see, for example,~\cite{JM} for a discussion of Lanchester's
laws in the context of human evolution.
Lanchester-type models are suitable not only for describing
human conflict but also have applications in
biology~\cite{AM, CLJA}.

One direction in the development of Lanchester models
involves incorporating multiple opposing forces.
The problem of trilateral conflict is studied under the
classical Lanchester model in~\cite{Kress-three} and under
the unaimed-fire model in~\cite{CCS}.
The latter also includes an analysis of the one-against-many
scenario.

Allocation problems related to combat models have also
been studied in the literature. These studies differ
considerably in their focus, assumptions, modeling
perspectives, and objective functions.
We briefly review the works most relevant to our setting.

The aforementioned one-against-many scenario---in which
the central force divides its attention among multiple
adversaries, while all adversaries attack the central
force as a whole---is studied in several settings.
These include the distribution of fire against multiple
opponents~\cite{JRF} and the incorporation of supply units
that can be targeted by the central force~\cite{NVHTD}.
In a related allocation setting, Sheeba and Ghose~\cite{SG}
study the initial allocation of a defending force against
two attacking forces, allowing redistribution after
decisive events.

The studies in~\cite{Taylor, TB} consider fire allocation
with dynamic adjustment.
In~\cite{Taylor}, the duration of the battle is prescribed,
and the objective function is a linear combination of
surviving friendly and enemy forces, with positive and
negative weights, respectively.
For two adversaries, combat continuing until one of
the forces is eliminated is also analyzed.
In~\cite{TB}, the setting involves two battlefields and
a single fire support unit that allocates its fire
between them, increasing the effectiveness of the
supported force.
The authors analyze how the choice of objective function
(a~weighted ratio or difference of survivors, or a weighted
ratio of casualties) affects optimal policies.
Chen and Zhang~\cite{CZ} study optimal control in a hybrid
dynamical model of combat between a~homogeneous force
and a heterogeneous force comprising multiple troop types.
In contrast to these dynamically adjusted allocation
policies, we consider initial allocations that remain
fixed throughout combat on independent battlefields,
with outcomes determined by elimination in finite time
or asymptotically.

Kaplan, Kress, and Szechtman~\cite{KKS} study force allocation
across insurgent strongholds under imperfect intelligence.
For a fixed set of strongholds to be attacked, they obtain
an allocation proportional to the victory thresholds.
Their broader objective accounts for both government
losses and civilian casualties.

MacKay~\cite{MacKay} considers a model of two armies,
each comprising an arbitrary number of unit types
that can use either aimed or unaimed fire.
The objective function for optimal allocation is derived
from conserved quantities. These quantities are obtained
under certain homogeneity conditions that preserve the
relative effectiveness of different unit types
(called separability in the paper).
They yield a natural cost-weighted ordering of the
opposing forces, with sequential destruction in this
order as an optimal strategy.

From an optimization perspective, our enemy loss
maximization problem is a continuous knapsack problem
with separable convex utilities.
Levi, Perakis, and Romero~\cite{LPR} establish NP-hardness
and develop approximation algorithms for this class
of problems.
Our NP-hardness result concerns prescribed families
of loss functions specified by integer saturation
thresholds, and our exact algorithm exploits the
integrality of these thresholds and the total force size.

We are not aware of previous results on optimal
allocation of units across independent battlefields 
for general classes of loss and survival functions, or even for the
power-law family considered in Section~\ref{general-model}.

    \section{Allocation through loss and survival functions}
\label{sec:allocation-framework}

\subsection{Setting and outcome functions}

We consider two opposing armies, red and green, engaged on $n$
independent battlefields. {Throughout the paper, we take the viewpoint of the red army:
``own'' refers to red units, and ``enemy'' refers to green units.} The initial numbers of green units
$G_1,\dots,G_n>0$ are fixed and known. The red army has $R>0$
units available and chooses an initial allocation
$R_1,\dots,R_n\ge0$. We treat units as a divisible resource,
so the allocations need not be integers. Once combat begins,
no units are transferred between battlefields, and the outcome
on each battlefield depends only on the allocation and the
parameters of that battlefield.

For each battlefield $i\in\{1,\dots,n\}$, let
\[
    k_i\colon[0,\infty)\to[0,G_i],
    \qquad
    s_i\colon[0,\infty)\to[0,\infty)
\]
denote the enemy loss function and the own survival function,
respectively. Thus, $k_i(x)$ is the number of green units
eliminated when $x$ red units are initially allocated to
battlefield $i$, while $s_i(x)$ is the number of surviving
red units. These quantities describe the outcome of the
engagement, whether reached in finite time or only
asymptotically. We assume
\[
    k_i(0)=s_i(0)=0,
    \qquad
    0\le s_i(x)\le x
    \quad\text{for all }x\ge0.
\]
The corresponding own loss function is
\[
    d_i(x)=x-s_i(x),\qquad x\ge0.
\]

We assume that each battlefield has a finite positive
\emph{saturation threshold} $R_i^*$ such that
\[
    k_i(x)<G_i\quad\text{for }0\le x<R_i^*,
    \qquad
    k_i(x)=G_i\quad\text{for }x\ge R_i^*.
\]
Thus, $R_i^*$ is the smallest allocation sufficient to
eliminate all green units on battlefield $i$. We call an
allocation to battlefield $i$ partially saturated if
$0<R_i<R_i^*$, and fully saturated if $R_i=R_i^*$.
Allocations above $R_i^*$ cannot increase enemy losses,
but may increase the number of red survivors.

Our allocation results are formulated directly in terms
of these functions and thresholds. Particular combat
models provide ways of determining the functions and thresholds, but the allocation
arguments do not require a specific system of differential
equations. Additional assumptions, such as convexity,
concavity, or superadditivity, will be imposed separately
for the individual optimization problems.

\subsection{Example: the classical Lanchester model}
\label{sec:lanchester-example}

To illustrate the preceding definitions, consider the
classical Lanchester model on a single battlefield:
\[
    R'(t)=-gG(t),\qquad
    G'(t)=-rR(t),\qquad
    R(0)=x,\quad G(0)=G,
\]
where $G>0$ is fixed, $x\ge0$ is the initial red allocation,
and $r,g>0$ describe the effectiveness of the red and green
armies, respectively. The differential equations apply
while both armies have positive sizes. For background
on this model and its variants, see
\cite{Clausen,BCD,Epstein,Kress}.

The quantity $rR(t)^2-gG(t)^2$ is conserved, since
\[
    \frac{\mathrm d}{\mathrm dt}
    \left(rR(t)^2-gG(t)^2\right)
    =2rR(t)R'(t)-2gG(t)G'(t)=0.
\]
Consequently,
\[
    rR(t)^2-gG(t)^2=rx^2-gG^2,
\]
and the sign of $rx^2-gG^2$ determines which army survives.
The saturation threshold is therefore
\[
    R^*=G\sqrt{\frac gr}.
\]

If $0\le x<R^*$, all red units are eliminated and the number
of green survivors is $\sqrt{G^2-(r/g)x^2}$. If $x>R^*$,
all green units are eliminated and the number of red
survivors is $\sqrt{x^2-(g/r)G^2}$. At $x=R^*$, both
armies tend to zero as $t\to\infty$. Hence the enemy loss
and own survival functions are
\[
    k(x)=
    \begin{cases}
        G-\sqrt{G^2-\dfrac rg x^2},
            &0\le x\le R^*,\\[2mm]
        G,&x>R^*,
    \end{cases}
\]
and
\[
    s(x)=
    \begin{cases}
        0,&0\le x\le R^*,\\[1mm]
        \sqrt{x^2-(R^*)^2},&x>R^*.
    \end{cases}
\]
The own loss function is $d(x)=x-s(x)$.

The function $k$ is increasing and strictly convex on
$[0,R^*]$, whereas $s$ is increasing and strictly concave
on $[R^*,\infty)$; these properties will be useful later.

For multiple battlefields, the same formulas apply with
parameters $G_i,r_i,g_i$ and thresholds
\[
    R_i^*=G_i\sqrt{\frac{g_i}{r_i}},
    \qquad i\in\{1,\dots,n\}.
\]
This gives a concrete instance of the general framework
introduced above.

\subsection{Allocation objectives}
\label{sec:allocation-objectives}

Our principal objective is to maximize total enemy losses:
\[
    \max\left\{
        \sum_{i=1}^n k_i(R_i):
        R_i\ge0,\quad
        \sum_{i=1}^n R_i\le R
    \right\}.
\]
Here, any unallocated red units may be retained in reserve.
Since $k_i$ is constant on $[R_i^*,\infty)$, replacing an
allocation $R_i>R_i^*$ by $R_i^*$ does not change enemy
losses. Thus, restricting allocations to $R_i\le R_i^*$
does not change the optimal value, and the problem becomes
\begin{equation}\label{eq:max-enemy-losses}
    \max\left\{
        \sum_{i=1}^n k_i(R_i):
        0\le R_i\le R_i^*,\quad
        \sum_{i=1}^n R_i\le R
    \right\}. \tag{$O_1$}
\end{equation}
If $R\ge\sum_{i=1}^n R_i^*$, every battlefield can be
saturated and the maximum equals $\sum_{i=1}^n G_i$.
The nontrivial case is therefore
$R<\sum_{i=1}^n R_i^*$.

We also consider several alternative objectives.
First, the red army may wish to maximize the number
of battlefields on which all green units are eliminated:
\begin{equation}\label{eq:max-successful-battles}
    \max\left\{
        \left|\{i\in\{1,\dots,n\}:R_i\ge R_i^*\}\right|:
        R_i\ge0,\quad
        \sum_{i=1}^n R_i\le R
    \right\}.\tag{$O_2$}
\end{equation}
Throughout the paper, success on a battlefield means
complete elimination of the opposing force, possibly
asymptotically; it does not require a positive number
of red survivors.

Second, suppose that all red units must be allocated
and the objective is to maximize their total number
of survivors:
\begin{equation}\label{eq:max-own-survivors}
    \max\left\{
        \sum_{i=1}^n s_i(R_i):
        R_i\ge0,\quad
        \sum_{i=1}^n R_i=R
    \right\}. \tag{$O_3$}
\end{equation}
Since
\[
    \sum_{i=1}^n d_i(R_i)
    =R-\sum_{i=1}^n s_i(R_i),
\]
this is equivalent to minimizing total red losses.
The requirement that all units be allocated excludes
the trivial strategy of avoiding all engagements.
A constrained version of this problem additionally
requires elimination of the green army on every
battlefield:
\begin{equation}\label{eq:max-survivors-all-battlefields}
    \max\left\{
        \sum_{i=1}^n s_i(R_i):
        R_i\ge R_i^*,\quad
        \sum_{i=1}^n R_i=R
    \right\}.\tag{$O_4$}
\end{equation}
This problem is feasible precisely when
$R\ge\sum_{i=1}^n R_i^*$.

Finally, we consider maximizing the ratio of total
enemy losses to total own losses, also called the
kill/death ratio:
\begin{equation}\label{eq:max-kd-ratio}
    \max\left\{
        \frac{\sum_{i=1}^n k_i(R_i)}
             {\sum_{i=1}^n d_i(R_i)}:
        R_i\ge0,\quad
        \sum_{i=1}^n R_i=R
    \right\}, \tag{$O_5$}
\end{equation}
under assumptions ensuring that the denominator
is positive.

These objectives can lead to different optimal allocations under the same battlefield outcome functions.
We first study enemy loss maximization, objective~\eqref{eq:max-enemy-losses},
including its computational complexity and an exact
algorithm, and then establish allocation rules for
the alternative objectives.

\section{Maximizing enemy losses:
structure, complexity, and an exact algorithm}
\label{sec:max-enemy-losses}

We consider the problem~\eqref{eq:max-enemy-losses}
introduced in Section~\ref{sec:allocation-objectives}.
Throughout this section, the restricted loss functions
\[
    k_i\colon[0,R_i^*]\to\mathbb R,
    \qquad i\in\{1,\dots,n\},
\]
are continuous, convex, and nondecreasing, with $k_i(0)=0$. These assumptions hold for the classical Lanchester model
introduced in Section~\ref{sec:lanchester-example}, as well
as for the mixed combat model discussed in
Section~\ref{sec:classical-combat-models}. More generally,
Section~\ref{sec:combat-applications} identifies the parameter
ranges of a power-law family~\eqref{eq:power-law-combat} for which the results below apply.

Convexity means that an additional unit produces at least
as much additional enemy loss when added to a larger
allocation as when added to a smaller one, provided that
neither allocation exceeds the saturation threshold.
This favors concentration, but saturation limits the
benefit of concentrating units on any one battlefield.

We first describe the structure of optimal allocations,
then establish NP-hardness under stronger assumptions,
and finally give an exact algorithm for integer input data.

\subsection{Structure of optimal allocations}
\label{sec:allocation-structure}

Convexity considerably restricts the allocations that
need to be considered.

\begin{theorem}\label{thm:allocation-structure}
Problem~\eqref{eq:max-enemy-losses} has an optimal
solution in which all but at most one component satisfy
\[
    R_i\in\{0,R_i^*\}.
\]
If every $k_i$ is strictly convex, every optimal solution
has this property.

Moreover, there exists an optimal solution satisfying
\[
    \sum_{i=1}^n R_i
    =\min\left(R,\sum_{i=1}^n R_i^*\right).
\]
If every $k_i$ is strictly increasing, every optimal
solution satisfies this equality.
\end{theorem}

\begin{proof}
The feasible set
\[
    P=\left\{
        (R_1,\dots,R_n):
        0\le R_i\le R_i^*,\quad
        \sum_{i=1}^n R_i\le R
    \right\}
\]
is a compact convex polytope, and the objective function
is continuous and convex. Its maximum is therefore
attained at a vertex of $P$.
At each vertex, at most one component $R_i$ can lie strictly between 0 and $R_i^*$.
Indeed, if
$0<R_i<R_i^*$ and $0<R_j<R_j^*$ for distinct indices
$i,j$, then sufficiently small transfers between these
two components in either direction preserve feasibility.
The original point is the midpoint of the resulting
distinct feasible points, so it is not a vertex.
This proves the existence assertion.

If every $k_i$ is strictly convex, their sum is strictly
convex. A point that is not a vertex is a nontrivial convex
combination of two distinct feasible points, and strict
convexity implies that it cannot maximize the objective.
Thus every optimal solution is a vertex and has the
stated property.

Finally, suppose that a feasible allocation satisfies
\[
    \sum_{i=1}^n R_i
    <\min\left(R,\sum_{i=1}^n R_i^*\right).
\]
At least one component can then be increased without
violating the constraints. Since the loss functions
are nondecreasing, we can increase components until
the displayed equality in the theorem holds without
decreasing the objective. If every loss function is
strictly increasing, any such increase strictly improves
the objective, so an optimal solution cannot leave
this amount of the total army unused.
\end{proof}

In fact, an optimal allocation can be chosen to satisfy
both conclusions simultaneously. If
$R\ge\sum_iR_i^*$, saturating every battlefield does so.
Otherwise, maximize the objective over the polytope
\[
    0\le R_i\le R_i^*,\qquad
    \sum_{i=1}^n R_i=R.
\]
By Theorem~\ref{thm:allocation-structure}, its maximum
equals that of the original problem. The same vertex
argument shows that it has an optimal solution with
at most one partially saturated battlefield.

The theorem describes the form of an optimal allocation,
but does not identify which battlefields should be
saturated. This remaining choice is the source of
computational difficulty.
Two special cases can be solved immediately:
\begin{itemize}
    \item If $R\ge\sum_{i=1}^n R_i^*$, saturating every
    battlefield gives the maximum
    $\sum_{i=1}^n k_i(R_i^*)=\sum_{i=1}^n G_i$.

    \item If $R\le\min_i R_i^*$, an optimal allocation
    sends all $R$ units to a battlefield maximizing
    $k_i(R)$. Indeed, the vertices of
    $\{(R_1,\dots,R_n):R_i\ge0,\ \sum_iR_i=R\}$
    are precisely these concentrated allocations.
\end{itemize}
Thus, the main difficulty concerns initial army sizes for which
some, but not all, battlefields can be saturated.

\begin{example}\label{ex:two-battlefields}
For two battlefields and
$0<R<R_1^*+R_2^*$, an optimal allocation can be found
on the line $R_1+R_2=R$. Its feasible part is a segment.
At one endpoint, we allocate as many units as possible
to battlefield 2: it receives all $R$ units if
$R\le R_2^*$, and otherwise receives $R_2^*$ units,
with the remaining $R-R_2^*$ units assigned to
battlefield 1. This gives the endpoint
\[
    \left(
        \max\{0,R-R_2^*\},
        \min\{R,R_2^*\}
    \right).
\]
The other endpoint is obtained by allocating as many
units as possible to battlefield 1 and assigning
the remainder to battlefield 2:
\[
    \left(
        \min\{R,R_1^*\},
        \max\{0,R-R_1^*\}
    \right).
\]
See Figure~\ref{fig:two-battlefields-scheme}.
Convexity implies that at least one of these two
allocations is optimal. Thus, for each fixed initial army size,
it suffices to compare at most two candidates.

As an example, consider the case $R_2^*\le R<R_1^*$.
The red army can eliminate all green units on
battlefield 2, but cannot do so on battlefield 1.
Nevertheless, allocating all units to battlefield~1
may inflict greater total enemy losses than saturating
battlefield 2 and assigning the remaining units to
battlefield 1. 
Thus, maximizing enemy losses may favor an allocation
that does not completely eliminate the opposing force
on either battlefield, even though one successful
engagement is feasible.

As the initial army size increases, the endpoints change when
$R$ passes one of the saturation thresholds.
Figure~\ref{fig:two-battlefields-scheme} illustrates
these changes, while
Figure~\ref{fig:two-battlefields-contours} shows the
objective contours for the classical Lanchester model.
The optimal allocation need not be unique, since the
two endpoints may give the same objective value.

\begin{figure}
    \centering
    \begin{subfigure}[b]{0.49\textwidth}
        \centering
        \includegraphics[width=\linewidth]
            {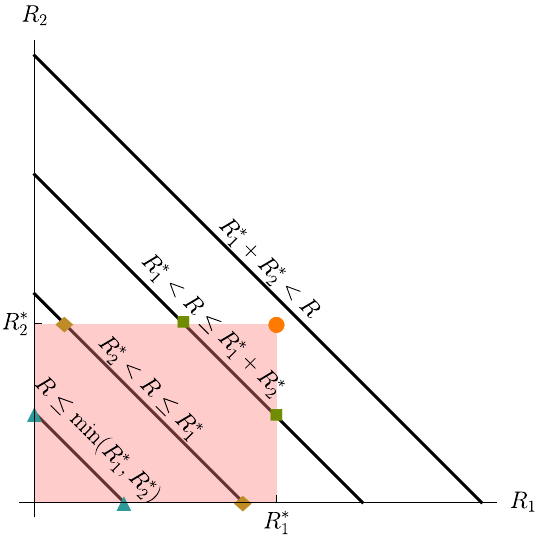}
        \caption{Candidate allocations for different
        initial army sizes relative to the saturation
        thresholds.}
        \label{fig:two-battlefields-scheme}
    \end{subfigure}
    \hfill
    \begin{subfigure}[b]{0.49\textwidth}
        \centering
        \includegraphics[width=\linewidth]
            {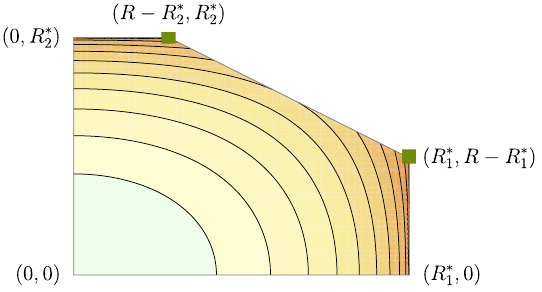}
        \caption{Objective contours for the classical
        Lanchester model. The candidates are the green square
        vertices.}
        \label{fig:two-battlefields-contours}
    \end{subfigure}
    \caption{The allocation problem with two battlefields
    from Example~\ref{ex:two-battlefields}.}
\end{figure}
\end{example}

For an arbitrary number of battlefields, the structural
result reduces the search to allocations determined
by a set of fully saturated battlefields and at most
one partially saturated battlefield. However, choosing
the set of fully saturated battlefields remains a
combinatorial problem, as the next subsection shows.

\subsection{NP-hardness}
\label{sec:NP-hard}

We establish NP-hardness for a restricted class of
instances of~\eqref{eq:max-enemy-losses} with normalized loss functions. The proof
uses a reduction from the Subset Sum Problem: given
positive integers $s_1,\dots,s_n$ and a positive target
integer $S$, determine whether there exists
$I\subseteq\{1,\dots,n\}$ such that
$\sum_{i\in I}s_i=S$.

Consider a prescribed family
$\{\ell_s\}_{s\in\mathbb N}$, where
$\ell_s\colon[0,s]\to\mathbb R$ is the enemy loss
function for a battlefield with saturation threshold
$s$. We impose the normalization
\[
    \ell_s(0)=0,\qquad \ell_s(s)=s.
\]
Thus, $s$ red units exactly suffice to eliminate
$s$ green units, whereas any smaller red allocation
leaves some green units alive. The family is fixed
in advance; each battlefield function is specified
by its positive integer index $s$.

\begin{theorem}\label{np-hard-problem}
Let $\{\ell_s\}_{s\in\mathbb N}$ be a prescribed family
of continuous, strictly convex, strictly increasing
functions $\ell_s\colon[0,s]\to\mathbb R$ satisfying
\[
    \ell_s(0)=0,\qquad \ell_s(s)=s.
\]
Then the following allocation problem is NP-hard:
given positive integers $R,R_1^*,\dots,R_n^*$ with
$R<\sum_{i=1}^n R_i^*$, determine the maximum of
\[
    \sum_{i=1}^n \ell_{R_i^*}(R_i)
\]
subject to
\[
    0\le R_i\le R_i^*,\qquad
    \sum_{i=1}^n R_i\le R.
\]
\end{theorem}

\begin{proof}
Consider a Subset Sum instance with positive integers
$s_1,\dots,s_n$ and target $S$. We may restrict attention
to instances satisfying $S<\sum_{i=1}^n s_i$, since
the remaining cases can be decided immediately.
Set
\[
    R_i^*=s_i,\qquad R=S,\qquad k_i=\ell_{s_i}.
\]
Since the family is prescribed and each function is
specified by its integer index, this construction
takes polynomial time.

Strict convexity and the endpoint normalization give
\[
    \ell_{s_i}(x)\le x,\qquad 0\le x\le s_i,
\]
with equality if and only if $x=0$ or $x=s_i$.
Consequently, every feasible allocation satisfies
\[
    \sum_{i=1}^n k_{i}(R_i)=\sum_{i=1}^n \ell_{s_i}(R_i)
    \le\sum_{i=1}^n R_i\le R=S.
\]

If the Subset Sum instance has a solution $I$, set
$R_i=s_i$ for $i\in I$ and $R_i=0$ otherwise.
This allocation is feasible and has objective value
$S$, so the maximum equals $S$.

Conversely, suppose that a feasible allocation
$(\bar R_1,\dots,\bar R_n)$ has objective value $S$.
Then
\[
    S=\sum_{i=1}^n \ell_{s_i}(\bar R_i)
    \le\sum_{i=1}^n\bar R_i\le S.
\]
Equality holds throughout. Hence
$\ell_{s_i}(\bar R_i)=\bar R_i$ for every $i$, which
implies $\bar R_i\in\{0,s_i\}$, and
$\sum_i\bar R_i=S$. The indices for which
$\bar R_i=s_i$ therefore solve the Subset Sum instance.

Thus the maximum equals $S$ if and only if the
Subset Sum instance has a solution, proving the claim.
\end{proof}

The proof follows the general idea of Proposition~1 in \cite{LPR}.
There, however, the utility functions form part of the input, and the
reduction uses specially chosen quadratic functions depending on the
instance. Here the family $\{\ell_s\}_{s\in\mathbb N}$ is fixed in
advance. Consequently, NP-hardness has to be established without
choosing the objective functions as part of the reduction; only the
integers $R,R_1^*,\dots,R_n^*$ may vary with the input instance.
This stronger formulation is needed in order to apply the result to
classical combat models, where the loss functions are prescribed by
the model rather than supplied as part of the input. For example, the
normalized classical Lanchester loss functions
\[
    \ell_s(x)=s-\sqrt{s^2-x^2},
    \qquad 0\le x\le s,
\]
satisfy the assumptions of the theorem. Further applications to
combat models will be considered below.

\subsection{An exact pseudopolynomial algorithm}
\label{sec:pseudopolynomial-algorithm}

Problem~\eqref{eq:max-enemy-losses} as considered in this section is a continuous
knapsack problem with separable convex utilities.
Levi, Perakis, and Romero~\cite{LPR} describe
approximation algorithms for this class of problems,
with approximation factors $\frac12$ and
$1-{\mathrm e}^{-1}$, respectively.
Here we give an exact algorithm when the initial army size
$R$ and the saturation thresholds $R_i^*$ are integers.

The allocation variables remain continuous.
Nevertheless, an integral optimal allocation exists:
by the structural result, an optimum can be chosen
that uses the entire army and has at most one
partially saturated component. Every fully saturated
component is integral, and the remaining component
equals the integer initial army size minus a sum of integer
thresholds.

The algorithm tries each possible partially saturated
battlefield $j$. Once $j$ is fixed, a knapsack dynamic
program chooses the fully saturated battlefields
among the remaining indices, and the unallocated
red units are assigned to $j$.

In the complexity analysis, evaluations of the loss
functions, additions, and comparisons are counted
as constant-time operations.

\begin{theorem}\label{pseudopolynomial-algo}
Let $R,R_1^*,\dots,R_n^*$ be positive integers such that
$R<\sum_{i=1}^n R_i^*$. Suppose that
$k_i\colon[0,R_i^*]\to\mathbb R$ are continuous,
convex, and nondecreasing, with $k_i(0)=0$.
Then an optimal allocation for
\eqref{eq:max-enemy-losses}, together with its
objective value, can be found using $O(n^2R)$
operations and $O(R)$ auxiliary space.
\end{theorem}

\begin{proof}
An optimal allocation can be chosen to use the entire
initial army size and to consist of a set of fully saturated
battlefields and at most one partially saturated
battlefield. We use $j\in\{1,\dots,n\}$ for the
possible partially saturated battlefield, and $j=0$
for an allocation with no partially saturated
battlefield.

For $j,k\in\{0,\dots,n\}$ and
$m\in\{0,\dots,R\}$, define
\[
    D_j(k,m)=
    \max\left\{
        \sum_{i\in S}k_i(R_i^*):
        S\subseteq\{1,\dots,k\}\setminus\{j\},
        \quad \sum_{i\in S}R_i^*=m
    \right\},
\]
where the maximum of an empty collection is
$-\infty$. Thus, $D_j(k,m)$ is the largest {total loss of the green
army achievable by using exactly $m$ red}  units to saturate
a subset of the first $k$ battlefields, excluding
battlefield $j$.

For each fixed $j$, initialize
\[
    D_j(0,0)=0,\qquad
    D_j(0,m)=-\infty\quad\text{for }m>0.
\]
The recurrence is
\[
D_j(k,m)=
\begin{cases}
\max\bigl\{
    D_j(k-1,m),
    D_j(k-1,m-R_k^*)+k_k(R_k^*)
\bigr\},
    & k\ne j,\ m\ge R_k^*,\\
D_j(k-1,m),
    & \text{otherwise}.
\end{cases}
\]
Indeed, a maximizing subset either omits battlefield
$k$ or includes it, in which case $R_k^*$ units are
used there and the objective increases by $k_k(R_k^*)$.

If $j\ne0$ and the saturated battlefields use $m$
units, the remaining $R-m$ units can be assigned to
battlefield $j$ whenever $0\le R-m\le R_j^*$.
The resulting total loss of the green army is
\[
    D_j(n,m)+k_j(R-m).
\]
The case with no partially saturated battlefield
has value $D_0(n,R)$. Consequently, the desired
maximum is
\begin{equation}\label{final-maximum}
\max\left\{
    D_0(n,R),
    \max_{1\le j\le n}
    \max_{\substack{m\in\{0,\dots,R\}\\
                    0\le R-m\le R_j^*}}
    \bigl(D_j(n,m)+k_j(R-m)\bigr)
\right\}.
\end{equation}
Every finite candidate in this expression corresponds
to a feasible allocation, and the structural result
guarantees that some candidate attains the optimum.

For a fixed $j$,  calculating the array $D_j$ using the recurrence requires $O(nR)$
operations. Running it for $j=0,\dots,n$ therefore
requires $O(n^2R)$ operations. 
The values in row $k$ depend only
on the values in the preceding row $k-1$, so it suffices to store two rows
of length $R+1$.
The tables $D_j$ for different values of $j$ can
be processed successively.

To recover an optimal allocation, retain a maximizing
pair $(j^*,m^*)$ in \eqref{final-maximum}, where
$m^*=R$ if $j^*=0$, together with the value
$D_{j^*}(n,m^*)$. We reconstruct a set of saturated
battlefields attaining this value by recomputing
the necessary rows.

Initially, set
\[
    m=m^*,\qquad V=D_{j^*}(n,m^*),
\]
and initialize the output allocation to zero.
For $k=n,n-1,\dots,1$, skip $k$ if $k=j^*$.
Otherwise, recompute $D_{j^*}(k-1,\cdot)$.
If
\[
    D_{j^*}(k-1,m)=V,
\]
leave battlefield $k$ empty. Otherwise, the recurrence
implies that an optimal subset includes battlefield
$k$; set $R_k=R_k^*$ and update
\[
    m:=m-R_k^*,\qquad
    V:=V-k_k(R_k^*).
\]
Finally, if $j^*\ne0$, set
\[
    R_{j^*}=R-m^*.
\]

Each recomputation requires $O(nR)$ operations and
$O(R)$ auxiliary space. There are at most $n$ such
recomputations, so reconstruction requires
$O(n^2R)$ operations and does not increase the overall
bound. The allocation is recorded directly in the
output vector. Thus the auxiliary space is $O(R)$,
in addition to the space for the input and output.
\end{proof}

The operation bound is pseudopolynomial because it
depends on the numerical value of $R$, rather than
only on the length of its binary representation.
The algorithm therefore does not contradict the
NP-hardness result.

\begin{remark}\label{rem:rational-allocation-data}
If $R,R_1^*,\dots,R_n^*$ are positive rationals, choose
a positive integer $Q$ such that
\[
    \widetilde R=QR,\qquad
    \widetilde R_i^*=QR_i^*
\]
are integers. With $\widetilde R_i=QR_i$, the problem
becomes
\[
    \max\left\{
        \sum_{i=1}^n k_i(\widetilde R_i/Q):
        0\le\widetilde R_i\le\widetilde R_i^*,
        \quad \sum_{i=1}^n\widetilde R_i
              \le\widetilde R
    \right\}.
\]
The functions $\widetilde k_i(x)=k_i(x/Q)$, defined
on $[0,\widetilde R_i^*]$, 
retain continuity,
convexity, and monotonicity. Hence the algorithm
applies, using $O(n^2QR)$ operations and $O(QR)$
auxiliary space, after which the allocation is
rescaled by dividing by $Q$.
\end{remark}

\section{Alternative objectives and optimal allocation rules}
\label{sec:alternative-objectives}

We now consider the alternative objectives introduced in
Section~\ref{sec:allocation-objectives}. Depending on the
objective and the properties of the outcome functions,
optimal allocations can be obtained by ordering saturation
thresholds, concentrating all units on a single battlefield,
or equalizing marginal survival values. 

\subsection{Maximizing the number of successful engagements}
\label{sec:max-successful-engagements}

Suppose that the objective is to maximize the number of
battlefields on which all green units are eliminated~\eqref{eq:max-successful-battles}.
Only the saturation thresholds matter for this problem:
battlefield $i$ is successful precisely when
$R_i\ge R_i^*$. An optimal allocation can therefore be
obtained by considering the battlefields in increasing
order of their thresholds.

\begin{theorem}\label{thm:max-successful-engagements}
Suppose that
\[
    R_1^*\le R_2^*\le\cdots\le R_n^*.
\]
The optimal value for the problem~\eqref{eq:max-successful-battles}
is
\[
    q=\max\left\{
        k\in\{0,\dots,n\}:
        \sum_{i=1}^k R_i^*\le R
    \right\}.
\]
An optimal allocation is obtained by setting
$R_i=R_i^*$ for $i\le q$ and $R_i=0$ for $i>q$.
\end{theorem}

\begin{proof}
The stated allocation uses at most $R$ units and gives
$q$ successful engagements. Conversely, any allocation
with $k$ successful engagements requires at least
$\sum_{i=1}^k R_i^*$ units, since this is the sum of
the $k$ smallest thresholds. Hence $k\le q$.
\end{proof}

After sorting the thresholds, the optimal number of
successful engagements is found by successively adding
them until the initial army size is exhausted.

\subsection{Maximizing own survivors}
\label{sec:max-own-survivors}

We next consider~\eqref{eq:max-own-survivors}. As explained in Section~\ref{sec:allocation-objectives},
this is equivalent to minimizing total red losses
when all $R$ units must be allocated.

Recall that a function $s\colon[0,R]\to\mathbb R$ is
superadditive if
\[
    s(x+y)\ge s(x)+s(y)
    \qquad\text{whenever }x,y\ge0,\quad x+y\le R.
\]
For a survival function, this means that combining
two allocations on the same battlefield yields at
least as many survivors as the sum of the survivors
from the two allocations considered separately.

\begin{theorem}\label{thm:survivor-concentration}
Let $s_i\colon[0,R]\to\mathbb R$ satisfy $s_i(0)=0$
for every $i$. Suppose that there exists
$j\in\{1,\dots,n\}$ such that $s_j$ is superadditive
and
\[
    s_j(x)\ge s_i(x)
    \qquad\text{for every }i\in\{1,\dots,n\}
    \text{ and }x\in[0,R].
\]
Then the allocation $R_j=R$ and $R_i=0$ for $i\ne j$
maximizes the total number of red survivors.
\end{theorem}

\begin{proof}
For every feasible allocation, pointwise dominance
and repeated application of superadditivity give
\[
    \sum_{i=1}^n s_i(R_i)
    \le\sum_{i=1}^n s_j(R_i)
    \le s_j\left(\sum_{i=1}^n R_i\right)
    =s_j(R).
\]
The allocation concentrating all units on battlefield
$j$ attains this upper bound.
\end{proof}

The following lemma gives sufficient conditions for
a survival function to be superadditive. In
particular, the conditions are satisfied by the classical Lanchester survival function
from Section~\ref{sec:lanchester-example}, as verified in
Section~\ref{sec:combat-applications}.

\begin{lemma}\label{lem:survival-superadditivity}
Suppose that $s\colon[0,\infty)\to\mathbb R$ satisfies
the following conditions:
\begin{enumerate}
    \item There exists $c\ge0$ such that $s(x)=0$
    for $x\in[0,c]$.
    \item The function $s$ is continuous and concave
    on $[c,\infty)$ and continuously differentiable
    on $(c,\infty)$.
    \item $\displaystyle\lim_{x\to\infty}s'(x)=1$.
    \item $s(x)\le x$ for every $x\ge c$.
\end{enumerate}
Then $s$ is superadditive on $[0,\infty)$.
\end{lemma}

\begin{proof}
Concavity implies that $s'$ is nonincreasing on $(c,\infty)$.
Together with its limit at infinity, this gives
$s'(x)\ge1$ for $x>c$. Consequently, $s$ is
nondecreasing on $[0,\infty)$.

If at least one of $x,y$ does not exceed $c$,
monotonicity and the first assumption give
\[
    s(x+y)\ge s(x)+s(y).
\]
If $x,y>c$, then
\[
    s(x+y)-s(x)
    =\int_x^{x+y}s'(u)\,{\mathrm d}u
    \ge \int_x^{x+y}1\,{\mathrm d}u =y
    \ge s(y).
\]
Thus $s$ is superadditive.
\end{proof}

The first assumption of the lemma describes a survival threshold:
no red units survive below it. The fourth assumption
excludes the creation of units during combat, while
the third says that, for any fixed $x>0$,
adding $x$ units to a very large army gives
approximately $x$ additional survivors.

A useful special case of
Theorem~\ref{thm:survivor-concentration} occurs when
the survival functions belong to a family $s_a$
indexed by a threshold $a>0$. If each $s_a$ is
superadditive and
\[
    a\le b\quad\Longrightarrow\quad
    s_a(x)\ge s_b(x)\quad\text{for every }x\ge0,
\]
then allocating all units to a battlefield with
the smallest threshold maximizes the number of
red survivors.

\subsection{Maximizing own survivors while eliminating
the opposing force on every battlefield}
\label{sec:min-loss-all-bf}

Suppose now that all green units must be eliminated.
We therefore consider~\eqref{eq:max-survivors-all-battlefields}.
If $R=\sum_iR_i^*$, the only feasible allocation is
$R_i=R_i^*$ for every $i$. In what follows, assume
that $R>\sum_iR_i^*$.

Unlike the preceding survivor maximization problem,
the present problem does not permit concentrating all
units on a single battlefield when $n>1$.
Each battlefield must receive at least its saturation
threshold. The remaining question is how to distribute
the additional units to maximize the total number
of survivors.

The derivative $s_i'(R_i)$ measures the rate at which
the number of red survivors increases when additional
units are allocated to battlefield $i$; we call it
the marginal survival value. If two battlefields
receive allocations above their lower bounds and
have different marginal survival values, transferring
a small number of units from the battlefield with
the lower value to the one with the higher value
increases the total number of survivors. This
suggests looking for an allocation satisfying
\[
    s_i'(R_i)=\lambda,\qquad i\in\{1,\dots,n\},
\]
for a common value $\lambda$.

\begin{theorem}
\label{allocation-max-survivors-and-conquer-all}
Suppose that $R>\sum_{i=1}^n R_i^*$ and that
$s_i\colon[R_i^*,\infty)\to\mathbb R$ are continuous,
strictly increasing, strictly concave, and
continuously differentiable on $(R_i^*,\infty)$.
Assume also that
\[
    \lim_{x\to R_i^*+}s_i'(x)=\infty
\]
for every $i$, and that there exists $m\ge0$ such that
\[
    \lim_{x\to\infty}s_i'(x)=m
\]
for every $i$.

Then the problem of maximizing
$\sum_{i=1}^n s_i(R_i)$ subject to
$R_i\ge R_i^*$ and $\sum_iR_i=R$ has a unique solution:
\begin{equation}\label{lambda-optimal}
    R_i=(s_i')^{-1}(\lambda),
    \qquad i\in\{1,\dots,n\},
\end{equation}
where $\lambda\in(m,\infty)$ is the unique solution of
\begin{equation}\label{lambda-equation}
    \sum_{i=1}^n(s_i')^{-1}(\lambda)=R.
\end{equation}
In particular, $R_i>R_i^*$ for every $i$.
\end{theorem}

\begin{proof}
Strict concavity and continuous differentiability of $s_i$ in $(R_i^*,\infty)$
imply that $s_i'$ is continuous and strictly decreasing
on $(R_i^*,\infty)$. By the assumed limits, it maps
this interval onto $(m,\infty)$.

Consequently, the function
\[
    F(\lambda)=\sum_{i=1}^n(s_i')^{-1}(\lambda)
\]
is continuous and strictly decreasing on $(m,\infty)$,
with
\[
    \lim_{\lambda\to m+}F(\lambda)=\infty,
    \qquad
    \lim_{\lambda\to\infty}F(\lambda)
    =\sum_{i=1}^n R_i^*<R.
\]
There is therefore a unique $\lambda\in(m,\infty)$
such that $F(\lambda)=R$. Set
$R_i=(s_i')^{-1}(\lambda)$. Then $R_i>R_i^*$ and
$\sum_iR_i=R$.

For any other feasible allocation
$(\widetilde R_1,\dots,\widetilde R_n)$,
concavity gives
\[
    s_i(\widetilde R_i)
    \le s_i(R_i)+s_i'(R_i)(\widetilde R_i-R_i)
    =s_i(R_i)+\lambda(\widetilde R_i-R_i).
\]
Summing these inequalities yields
\[
    \sum_{i=1}^n s_i(\widetilde R_i)
    \le\sum_{i=1}^n s_i(R_i),
\]
because both allocations have total size $R$.
Strict concavity makes the inequality strict
whenever the allocations differ, proving uniqueness.
\end{proof}

If the survival functions share a common form after
scaling by their thresholds, the optimal allocation
can be written explicitly.

\begin{corollary}\label{cor:proportional-allocation}
Let $R_1^*,\dots,R_n^*>0$ and
$R>\sum_{i=1}^n R_i^*$.
Suppose that $h\colon[1,\infty)\to\mathbb R$ is
continuous, strictly increasing, strictly concave,
and continuously differentiable on $(1,\infty)$, with
\[
    \lim_{x\to1+}h'(x)=\infty,
    \qquad
    \lim_{x\to\infty}h'(x)=m\in[0,\infty).
\]
For the survival functions
\[
    s_i(x)=R_i^*h\left(\frac{x}{R_i^*}\right),
    \qquad x\ge R_i^*,
\]
the unique allocation maximizing
$\sum_{i=1}^n s_i(R_i)$ subject to
$R_i\ge R_i^*$ and $\sum_iR_i=R$ is
\[
    R_i=\frac{R_i^*}{\sum_{j=1}^n R_j^*}R,
    \qquad i\in\{1,\dots,n\}.
\]
\end{corollary}

\begin{proof}
The survival functions satisfy the assumptions of
Theorem~\ref{allocation-max-survivors-and-conquer-all}.
Since
\[
    s_i'(x)=h'\left(\frac{x}{R_i^*}\right),
\]
we have
\[
    (s_i')^{-1}(\lambda)
    =R_i^*(h')^{-1}(\lambda).
\]
Consequently, equation~\eqref{lambda-equation} becomes
\[
    (h')^{-1}(\lambda)\sum_{i=1}^n R_i^*=R,
\]
and hence
\[
    (h')^{-1}(\lambda)
    =\frac{R}{\sum_{i=1}^n R_i^*}.
\]
Substitution into~\eqref{lambda-optimal} gives
\[
    R_i=R_i^*(h')^{-1}(\lambda)
    =\frac{R_i^*}{\sum_{j=1}^n R_j^*}R,
    \qquad i\in\{1,\dots,n\}.\tag*{\qedhere}
\]
\end{proof}

Hence, under the common scaling assumption, the
optimal allocation gives every battlefield the
same ratio of allocated units to its saturation
threshold.

\subsection{Maximizing the kill/death ratio}
\label{sec:max-kd-ratio}

Here we consider~\eqref{eq:max-kd-ratio}.
Recall that $k_i(x)$ and $d_i(x)$ denote enemy losses
and own losses on battlefield~$i$. For $x>0$, define
the individual kill/death ratio by
\[
    \kappa_i(x)=\frac{k_i(x)}{d_i(x)}.
\]
The global ratio is
\[
    \kappa(R_1,\dots,R_n)
    =\frac{\sum_{i=1}^n k_i(R_i)}
           {\sum_{i=1}^n d_i(R_i)}.
\]
We assume that $k_i(0)=d_i(0)=0$ and that
$k_i(x),d_i(x)>0$ for $x\in(0,R]$.

For any allocation with $\sum_iR_i=R>0$, the global
ratio is a weighted average of the individual ratios
on the active battlefields:
\[
    \kappa(R_1,\dots,R_n)
    =
    \sum_{i:R_i>0}
    \frac{d_i(R_i)}{\sum_{j=1}^n d_j(R_j)}
    \,\kappa_i(R_i).
\]
The weights are positive and sum to one. In particular,
\[
    \kappa(R_1,\dots,R_n)
    \le\max_{i:R_i>0}\kappa_i(R_i),
\]
which is an instance of Cauchy's mediant inequality.

\begin{theorem}\label{thm:kd-ratio}
Let $R>0$, and assume that $k_i(0)=d_i(0)=0$ and
$k_i(x),d_i(x)>0$ for $x\in(0,R]$. Define
\[
    J=\left\{
        j\in\{1,\dots,n\}:
        \kappa_j(R)=\max_{1\le i\le n}\kappa_i(R)
    \right\}.
\]
If every $\kappa_i$ is nondecreasing on $(0,R]$,
then allocating all $R$ units to any battlefield
$j\in J$ maximizes the global kill/death ratio under
the constraints
\[
    R_i\ge0,\qquad \sum_{i=1}^n R_i=R.
\]
If every $\kappa_i$ is strictly increasing, these
are the only optimal allocations.
\end{theorem}

\begin{proof}
For any feasible allocation, the weighted-average
identity and monotonicity give
\[
    \kappa(R_1,\dots,R_n)
    \le\max_{i:R_i>0}\kappa_i(R_i)
    \le\max_{1\le i\le n}\kappa_i(R).
\]
If $j\in J$, the allocation $R_j=R$ and $R_i=0$
for $i\ne j$ attains this upper bound.

Now suppose that every $\kappa_i$ is strictly
increasing and that an allocation uses at least
two battlefields. Then $R_i<R$ for every active
battlefield, so
\[
    \kappa(R_1,\dots,R_n)
    \le\max_{i:R_i>0}\kappa_i(R_i)
    <\max_{1\le i\le n}\kappa_i(R).
\]
Such an allocation cannot be optimal. Therefore
every optimal allocation concentrates all units
on a single battlefield, which must belong to $J$.
\end{proof}

\begin{remark}
If all green units must be eliminated, the numerator
of the global ratio is fixed:
\[
    \sum_{i=1}^n k_i(R_i)=\sum_{i=1}^n G_i.
\]
Provided that total red losses are positive,
maximizing the kill/death ratio is therefore
equivalent to minimizing total red losses.
With all $R$ units allocated, this is precisely
the survivor maximization problem considered in
Section~\ref{sec:min-loss-all-bf}.
\end{remark}

Table~\ref{tab:allocation-rules} summarizes the allocation
results from Sections~\ref{sec:max-enemy-losses}
and~\ref{sec:alternative-objectives}.
The assumptions listed are sufficient conditions;
the precise hypotheses are given in the corresponding
theorems.

\begin{table}[htbp]
\centering
\caption{Objectives, sufficient conditions, and optimal
allocation rules.}
\label{tab:allocation-rules}
\small
\renewcommand{\arraystretch}{1.3}
\begin{tabular}{
    @{}p{0.29\textwidth}
    p{0.25\textwidth}
    p{0.39\textwidth}@{}
}
\hline
Objective & Sufficient conditions & Optimal allocation rule\\
\hline

\eqref{eq:max-enemy-losses}: Maximize enemy losses
&
Continuous, nondecreasing convex loss functions
&
An optimum has at most one partially saturated
battlefield; choosing the saturated battlefields
remains a combinatorial problem.
\\

\eqref{eq:max-successful-battles}: Maximize the number of successful engagements
&
No additional assumptions
&
Saturate battlefields in nondecreasing order
of their thresholds.
\\

\eqref{eq:max-own-survivors}: Maximize own survivors
&
One superadditive survival function dominates
all the others pointwise
&
Concentrate all units on a battlefield with
the dominating survival function.
\\

\eqref{eq:max-survivors-all-battlefields}: Maximize own survivors while eliminating all
opposing forces
&
Strictly concave survival functions satisfying
the derivative conditions of
Theorem~\ref{allocation-max-survivors-and-conquer-all} and sufficient initial army size
&
Equalize marginal survival values.
Under the common scaling assumption of
Corollary~\ref{cor:proportional-allocation},
allocate proportionally to the thresholds.
\\

\eqref{eq:max-kd-ratio}: Maximize the kill/death ratio
&
Positive losses for positive allocations;
nondecreasing individual ratios
&
Concentrate all units on a battlefield maximizing
$\kappa_i(R)$. If all individual ratios are strictly
increasing, only such allocations are optimal.
\\

\hline
\end{tabular}
\end{table}

\section{Applications to combat models}
\label{sec:combat-applications}

We now apply the preceding allocation results to several
combat models. We begin with a power-law family containing
the classical Lanchester, guerrilla, and mixed combat models.
To illustrate the generality of our results, we then consider models involving an intelligence function
and implicitly defined battlefield outcomes.

Throughout this section, the differential equations describe
combat while both armies have positive sizes. Outcomes include
both elimination in finite time and asymptotic elimination.

\subsection{A power-law family of combat models}
\label{general-model}

Consider the model
\begin{equation}\label{eq:power-law-combat}
    R'(t)=-gG(t)^\alpha R(t)^\beta,\qquad
    G'(t)=-rR(t)^\gamma G(t)^\delta,
    \qquad R(0)=x,\quad G(0)=G,
\end{equation}
where $r,g,G>0$ and
$\alpha,\beta,\gamma,\delta\ge0$; see, for example,
\cite{Epstein,MacKay}. The variable $x$ denotes the initial
red allocation. Define the effective exponents
\[
    A=\gamma-\beta+1,\qquad
    B=\alpha-\delta+1.
\]
We assume $A,B>0$, which gives finite positive saturation
thresholds, as we show later. On the other hand, we do not suppose that $A$, $B$ are integers. 
Combat models with noninteger effective exponents already appear in early
literature, see e.g.~Osipov~\cite[eq.~(12)]{Osipov}, who considered
a relation between initial and remaining force sizes
corresponding to $A=B=3/2$ and $r=g$.

The conserved quantity is
\[
    \frac rA R(t)^A-\frac gB G(t)^B.
\]
Indeed,
\[
    \frac{\mathrm d}{\mathrm dt}
    \left(\frac rA R(t)^A-\frac gB G(t)^B\right)
    =rR(t)^{A-1}R'(t)-gG(t)^{B-1}G'(t)=0.
\]

Both army sizes are nonincreasing, and their terminal
values cannot both be positive, as follows directly
from the differential equations.
Set
\[
    c=\frac BA\frac rg.
\]
The conservation law gives
\[
    G(t)^B-cR(t)^A=G^B-cx^A.
\]
Consequently, the saturation threshold is
\[
    R^*=\left(\frac{G^B}{c}\right)^{1/A},
\]
and the enemy loss function is
\begin{equation}\label{eq:power-law-loss}
    k(x)=
    \begin{cases}
        G-(G^B-cx^A)^{1/B},&0\le x\le R^*,\\
        G,&x>R^*.
    \end{cases}
\end{equation}
The own survival function is
\begin{equation}\label{eq:power-law-survival}
    s(x)=
    \begin{cases}
        0,&0\le x\le R^*,\\
        (x^A-(R^*)^A)^{1/A},&x>R^*.
    \end{cases}
\end{equation}
The four exponents in the differential equations enter
the outcome formulas only through the combinations
$A$ and $B$. Different systems with the same effective
exponents and effectiveness ratio therefore have the
same loss and survival functions, although the time
evolution of combat may differ.

\begin{remark}
The restriction $A,B>0$ excludes unrealistic outcomes. If $A<0<B$, conservation prevents the red army
from tending to zero, so any positive red allocation
eventually eliminates the green army. If $B<0<A$,
the green army survives regardless of the finite
positive red allocation. If $A,B<0$, both armies
tend to zero. These cases do not provide the finite
positive saturation thresholds used in our framework.
The cases $A=0$ or $B=0$ require logarithmic conserved
quantities and are not considered here.
\end{remark}

\subsubsection{Convexity and enemy loss maximization, \eqref{eq:max-enemy-losses}}

The following result characterizes when the loss function
has the convexity required in
Section~\ref{sec:max-enemy-losses}.

\begin{theorem}\label{thm:power-law-convexity}
Let $A,B>0$. The loss function $k$ in
\eqref{eq:power-law-loss} is strictly increasing on
$[0,R^*]$. Its restriction to this interval is convex
if and only if
\[
    A\ge1,\qquad B\ge1.
\]
It is strictly convex if and only if, in addition,
$(A,B)\ne(1,1)$.
\end{theorem}

\begin{proof}
Strict monotonicity follows directly from the formula~\eqref{eq:power-law-loss}.

Suppose first that $A,B\ge1$. The function
$x\mapsto G^B-cx^A$ is concave and nonnegative on
$[0,R^*]$, and the function $u\mapsto u^{1/B}$ is
increasing and concave on $[0,\infty)$. Their composition
is therefore concave, so $k$ is convex.

If $A>1$, the inner function is strictly concave,
and the composition is strictly concave because
the outer function is increasing and concave.
If $A=1$ and $B>1$, the inner function is affine
and nonconstant, while the outer function is
strictly concave. Thus $k$ is strictly convex
in both cases. For $A=B=1$, we have $k(x)=cx$,
which is linear.

To prove necessity, differentiation on $(0,R^*)$
gives
\[
    k''(x)=
    \frac{cA}{B}x^{A-2}(G^B-cx^A)^{1/B-2}
    \left((A-1)G^B+
          \left(1-\frac AB\right)cx^A\right).
\]
The factors outside the parentheses are positive.
If $A<1$, the expression in parentheses is negative
for sufficiently small positive $x$. If $B<1$,
its limit as $x\to R^*-$ is
\[
    A\left(1-\frac1B\right)G^B<0.
\]
In either case, $k$ is not convex on $[0,R^*]$.
\end{proof}

Figure~\ref{fig:phi} illustrates how the effective
exponents $A$ and $B$ influence the shape of the
enemy loss function below saturation.

\begin{figure}[ht]
    \centering
    \includegraphics[width=0.7\linewidth]{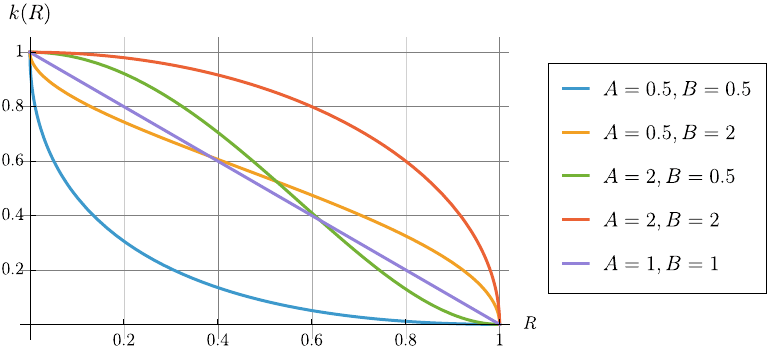}
    \caption{The enemy loss function
    $k(x)=1-(1-x^A)^{1/B}$ on $[0,1]$.
    The thresholds $A=1$ and $B=1$ determine convexity.
    The exponent $A$ controls the behavior near zero,
    while $B$ controls the behavior near saturation.
    The classical Lanchester and guerrilla models
    correspond to $(A,B)=(2,2)$ and $(1,1)$,
    respectively.}
    \label{fig:phi}
\end{figure}

The convexity conditions are equivalent to
$\gamma\ge\beta$ and $\alpha\ge\delta$.
Strict convexity requires at least one of these
inequalities to be strict.

\begin{theorem}\label{thm:power-law-allocation}
Fix nonnegative exponents satisfying
\[
    \gamma\ge\beta,\qquad \alpha\ge\delta,
\]
with at least one inequality strict. Consider
$n$ battlefields governed by
\eqref{eq:power-law-combat}, with these common exponents
and positive parameters $r_i,g_i,G_i$.

The problem of maximizing total enemy losses,~\eqref{eq:max-enemy-losses}, is NP-hard. If the initial army sizes $R$ and the saturation
thresholds $R_1^*,\dots,R_n^*$ are integers, an optimal
allocation can be found using $O(n^2R)$ operations
and $O(R)$ auxiliary space.
\end{theorem}

\begin{proof}
The effective exponents satisfy $A,B\ge1$ and
$(A,B)\ne(1,1)$. For each positive integer $q$, choose
\[
    G=q,\qquad g=1,\qquad
    r=\frac AB q^{B-A}.
\]
Here the exponents are fixed, and each normalized battlefield
instance is specified by the positive integer $q$; the parameters
above are determined by this index.
Then $c=q^{B-A}$ and $R^*=q$. The corresponding loss functions
belong to the family
\[
    \ell_q(x)=
    q\left[
        1-\left(1-\left(\frac{x}{q}\right)^A\right)^{1/B}
    \right],
    \qquad 0\le x\le q.
\]
Each function is continuous, strictly increasing,
and strictly convex, with
$\ell_q(0)=0$ and $\ell_q(q)=q$.
Theorem~\ref{np-hard-problem} therefore applies.

For arbitrary positive battlefield parameters,
Theorem~\ref{thm:power-law-convexity} gives the
assumptions of Theorem~\ref{pseudopolynomial-algo}.
If $R\ge\sum_iR_i^*$, saturating all battlefields
is optimal. Otherwise, the claimed operation
and space bounds follow from that theorem.
\end{proof}

If $A=B=1$, the restricted loss functions are linear:
\[
    k_i(x)=c_i x,\qquad 0\le x\le R_i^*.
\]
In this case, enemy loss maximization is a continuous
linear knapsack problem. It is solved by ordering
the coefficients $c_i$ in nonincreasing order and
saturating battlefields in that order until the
initial army size is exhausted.

\subsubsection{Survivor maximization, \eqref{eq:max-own-survivors} and~\eqref{eq:max-survivors-all-battlefields}}

For $A\ge1$, the survival functions
\[
    s_a(x)=
    \begin{cases}
        0,&0\le x\le a,\\
        (x^A-a^A)^{1/A},&x>a
    \end{cases}
\]
are superadditive. Indeed, $s_a(x)\le x$, and for
$x>a$,
\[
    s_a'(x)
    =\left(1-\left(\frac ax\right)^A\right)^{1/A-1}.
\]
For $A>1$, this derivative is strictly decreasing
from infinity to one, so $s_a$ is strictly concave
on $[a,\infty)$. For $A=1$, it is linear there.
Thus Lemma~\ref{lem:survival-superadditivity} applies
in both cases.

Moreover, $s_a(x)$ is nonincreasing with respect to
$a$. Consequently,
Theorem~\ref{thm:survivor-concentration} gives the
following conclusion.

\begin{corollary}\label{cor:power-law-survivors}
For the power-law model~\eqref{eq:power-law-combat} with common exponent $A\ge1$,
allocating all red units to a battlefield with the
smallest saturation threshold maximizes the total
number of red survivors.
\end{corollary}

If all green units must be eliminated, the allocation
rule changes. For $A>1$, write
\[
    s_i(x)=R_i^*h\left(\frac{x}{R_i^*}\right),
    \qquad
    h(u)=(u^A-1)^{1/A},\quad u\ge1.
\]
The function $h$ satisfies the assumptions of
Corollary~\ref{cor:proportional-allocation}.
Hence, when $R>\sum_iR_i^*$, the unique optimal
allocation is
\[
    R_i=\frac{R_i^*}{\sum_{j=1}^n R_j^*}R,
    \qquad i\in\{1,\dots,n\}.
\]
For $A=1$, every feasible allocation with
$R_i\ge R_i^*$ gives the same total number of survivors,
namely $R-\sum_iR_i^*$.

\subsubsection{The kill/death ratio}

The own loss function associated with
\eqref{eq:power-law-survival} is
\[
    d(x)=
    \begin{cases}
        x,&0\le x\le R^*,\\
        x-(x^A-(R^*)^A)^{1/A},&x>R^*.
    \end{cases}
\]
Using \eqref{eq:power-law-loss} and $G=c^{1/B}(R^*)^{A/B}$, we calculate the individual
kill/death ratio for $x>0$, namely
\begin{equation}\label{eq:kappa}
    \kappa(x)=\frac{k(x)}{d(x)}=
    \begin{cases}
    \dfrac{c^{1/B}
        \left[(R^*)^{A/B}
        -((R^*)^A-x^A)^{1/B}\right]}{x},
        &0<x\le R^*,\\[3mm]
    \dfrac{c^{1/B}(R^*)^{A/B}}
        {x-(x^A-(R^*)^A)^{1/A}},
        &x>R^*.
    \end{cases}
\end{equation}
For convenience, set $\kappa(0)=0$; continuity at
zero is not required.

\begin{lemma}\label{lem:power-law-kd}
Let $A,B>0$. The function $\kappa$ is nondecreasing
on $[0,\infty)$ if and only if $A,B\ge1$.
It is strictly increasing on $[0,\infty)$ if and
only if $A>1$ and $B\ge1$.
\end{lemma}

\begin{proof}
Suppose first that $A,B\ge1$. By
Theorem~\ref{thm:power-law-convexity}, the restricted enemy
loss function $k$ is convex and $k(0)=0$.
For $0<x<y\le R^*$, convexity gives
\[
    \frac{k(x)}{x}\le \frac{k(y)}{y}.
\]
Thus $\kappa$ is
nondecreasing on $(0,R^*]$. If $A>1$, strict
convexity makes this inequality strict.

For $x>R^*$, the numerator of $\kappa$ is constant,
and differentiation gives
\[
    d(x)^2\kappa'(x)
    =
    G\left[
        \left(1-\left(\frac{R^*}{x}\right)^A
        \right)^{-1+1/A}-1
    \right].
\]
Hence $\kappa'$ is negative when $A<1$, zero when
$A=1$, and positive when $A>1$.
The two branches agree at $R^*$, and
$\kappa(x)>0=\kappa(0)$ for every $x>0$.
This proves sufficiency and shows that
nondecreasing behavior requires $A\ge1$, while
strictly increasing behavior requires $A>1$.

Finally, if $0<B<1$, differentiation of the first
branch gives
\[
    \kappa'(x)=\frac{c^{1/B}}{x^2}
    \left[
        -(R^*)^{A/B}
        +((R^*)^A-x^A)^{1/B}
        +\frac AB x^A((R^*)^A-x^A)^{1/B-1}
    \right].
\]
Therefore,
\[
    \lim_{x\to {R^*}^-}\kappa'(x)
    =-c^{1/B}(R^*)^{A/B-2}<0.
\]
Thus $\kappa$ cannot be nondecreasing when $B<1$.
\end{proof}

Combining this lemma with Theorem~\ref{thm:kd-ratio}
shows that, for $A,B\ge1$, an optimal allocation
concentrates all units on a battlefield maximizing
$\kappa_i(R)$. If $A>1$, these are the only optimal
allocations.

Notice the distinction between the conditions for
strict convexity and strict monotonicity of the
ratio. When $A=1$ and $B>1$, the loss function is
strictly convex below saturation, but the ratio
is constant above saturation.

\subsection{Classical models as special cases}
\label{sec:classical-combat-models}

The preceding results include the classical
Lanchester, guerrilla, and mixed combat models.
Their exponents $\alpha$, $\beta$, $\gamma$, $\delta$ and effective exponents $A$, $B$ are summarized below.
\[
\begin{array}{c|c|c}
\text{Model}
&(\alpha,\beta,\gamma,\delta)&(A,B)\\ \hline
\text{Classical Lanchester}&(1,0,1,0)&(2,2)\\
\text{Guerrilla}&(1,1,1,1)&(1,1)\\
\text{Mixed, regular allocation}&(1,0,1,1)&(2,1)\\
\text{Mixed, guerrilla allocation}&(1,1,1,0)&(1,2)
\end{array}
\]
In the last row, the armies have been relabeled so
that the allocating army is again called red.

\subsubsection{The classical Lanchester model}

The classical model was introduced in
Section~\ref{sec:lanchester-example}. Its threshold
and restricted loss function are
\[
    R_i^*=G_i\sqrt{\frac{g_i}{r_i}},
    \qquad
    k_i(x)=G_i-\sqrt{G_i^2-\frac{r_i}{g_i}x^2}.
\]
Since $A=B=2$, the restricted loss functions are
strictly convex.

\begin{corollary}
The problem~\eqref{eq:max-enemy-losses} of maximizing total enemy losses in
the classical Lanchester model is NP-hard.

If $R,R_1^*,\dots,R_n^*$ are integers, an optimal
allocation can be found using $O(n^2R)$ operations
and $O(R)$ auxiliary space.
\end{corollary}

\begin{proof}
When $r_i=g_i=1$, the threshold is $R_i^*=G_i$,
and the loss functions belong to the normalized
family
\[
    \ell_q(x)=q-\sqrt{q^2-x^2}.
\]
The hardness claim follows from
Theorem~\ref{np-hard-problem}, and the algorithmic
claim from Theorem~\ref{pseudopolynomial-algo}.
\end{proof}

For the alternative objectives, all the preceding
power-law conclusions apply. In particular, survivor
maximization,~\eqref{eq:max-own-survivors}, favors concentration on a battlefield
with the smallest threshold, while survivor
maximization subject to eliminating all green units,~\eqref{eq:max-survivors-all-battlefields}, 
gives the proportional allocation. The individual
kill/death ratios are strictly increasing, so every
allocation maximizing the global kill/death ratio,~\eqref{eq:max-kd-ratio}, concentrates
all units on one battlefield.

\subsubsection{The guerrilla combat model}
\label{sec:guerilla}

The guerrilla model is
\[
    R'(t)=-gR(t)G(t),\qquad
    G'(t)=-rR(t)G(t).
\]
It describes two dispersed forces searching for each
other, with attrition depending on encounters between
them; see \cite{Clausen,BCD,Epstein}.
Its conserved quantity is $rR(t)-gG(t)$.
Neither army is eliminated in finite time.

Here $A=B=1$, and
\[
    R^*=\frac gr G,\qquad
    k(x)=\frac rg x\quad(0\le x\le R^*),
    \qquad
    s(x)=\max\{0,x-R^*\}.
\]
Thus enemy loss maximization,~\eqref{eq:max-enemy-losses}, is solved by sorting
the effectiveness ratios $r_i/g_i$ in nonincreasing
order and saturating battlefields in that order.

Allocating all units to a battlefield with the
smallest threshold maximizes red survivors,~\eqref{eq:max-own-survivors}.
If all green units must be eliminated,~\eqref{eq:max-survivors-all-battlefields}, every feasible
allocation has the same total number of red survivors.

The individual kill/death ratio is constant for
positive allocations:
\[
    \kappa_i(x)=\frac{r_i}{g_i},\qquad x>0.
\]
Consequently, the global kill/death ratio is maximized,~\eqref{eq:max-kd-ratio}, by
allocating units only to battlefields with the
largest effectiveness ratio. Any distribution
among such battlefields is optimal.

\subsubsection{The mixed combat model}

The mixed model, also known as Deitchman's guerrilla warfare
model~\cite{Deitchman}, describes a regular red army
confronting dispersed green guerrillas:
\[
    R'(t)=-gG(t),\qquad
    G'(t)=-rR(t)G(t).
\]
Regular forces (red army) use area fire, whose effectiveness
depends on both their own size and the density
of guerrillas. Guerrillas (green army) can locate the regular
forces and use aimed fire. The conserved quantity is
\[
    rR(t)^2-2gG(t).
\]
If the regular force is defeated while guerrillas survive,
it is eliminated in finite time. Elimination of the guerrillas
always occurs asymptotically, including the case in which
both forces tend to zero.

For allocation of regular forces, $A=2$ and $B=1$,
so
\[
    R^*=\sqrt{\frac{2g}{r}G},
    \qquad
    k(x)=\frac{r}{2g}x^2\quad(0\le x\le R^*),
\]
and
\[
    s(x)=
    \begin{cases}
        0,&0\le x\le R^*,\\
        \sqrt{x^2-(R^*)^2},&x>R^*.
    \end{cases}
\]
Enemy loss maximization,~\eqref{eq:max-enemy-losses} is NP-hard. Explicitly,
choosing
\[
    G=q,\qquad g=1,\qquad r=\frac2q
\]
gives $R^*=q$ and the normalized family
\[
    \ell_q(x)=\frac{x^2}{q},\qquad 0\le x\le q.
\]
The exact algorithm applies when the initial army size and
thresholds are integers. The  allocation
rules for survivor maximization,~\eqref{eq:max-own-survivors},~\eqref{eq:max-survivors-all-battlefields}, are the same as for the classical model,
and the kill/death ratios are strictly increasing,~\eqref{eq:max-kd-ratio}.

Now reverse the allocation problem: fix the initial
regular forces $R_i$ and distribute a total of $G$
guerrilla units. For a single battlefield with
fixed regular force $R$, the guerrilla threshold is
\[
    G^*=\frac{r}{2g}R^2,
\]
and the losses inflicted on the regular army by
a guerrilla allocation $y$ are
\[
    R-\sqrt{R^2-\frac{2g}{r}y},
    \qquad 0\le y\le G^*.
\]
After interchanging the roles of the armies,
the effective exponents are $A=1$ and $B=2$.
Thus the restricted loss function is again
strictly convex. Choosing
\[
    R=q,\qquad g=1,\qquad r=\frac2q
\]
gives $G^*=q$ and the normalized family
\[
    \ell_q(y)=q-\sqrt{q^2-qy},
    \qquad 0\le y\le q.
\]
Hence this allocation problem,~\eqref{eq:max-enemy-losses}, is also NP-hard.
For integral $G,G_1^*,\dots,G_n^*$, it is solvable
using $O(n^2G)$ operations and $O(G)$ auxiliary space.

The guerrilla survival function is
$\max\{0,y-G^*\}$. Consequently, allocating all
guerrillas to a battlefield with the smallest
threshold maximizes their survivors,~\eqref{eq:max-own-survivors}. If every
regular force must be eliminated,~\eqref{eq:max-survivors-all-battlefields}, every feasible
allocation gives the same total number of
guerrilla survivors. The individual kill/death
ratios are nondecreasing, but become constant
above their thresholds,~\eqref{eq:max-kd-ratio}.

\subsection{An intelligence-dependent combat model}
\label{sec:intelligence-model}

Kaplan, Kress, and Szechtman~\cite{KKS} consider
government forces confronting guerrillas dispersed
within a stronghold. Their model has the form
\[
    R'(t)=-gG(t),\qquad
    G'(t)=-rR(t)p(G(t)),
    \qquad R(0)=x,\quad G(0)=G,
\]
where $p$ is an intelligence function representing
the per-shot probability of targeting an insurgent.

For the following calculation, assume that $p$ is
continuous and positive on $(0,G]$ and that
\[
    H(G) < \infty,
\]
where
\[
    H(v):=\int_0^v\frac{u}{p(u)}\,{\mathrm d}u,\quad v\in[0,G].
\]
The conserved quantity is
\[
    \frac r2 R(t)^2-gH(G(t)),
\]
because
\[
    \frac{\mathrm d}{\mathrm dt}
    \left(\frac r2 R(t)^2-gH(G(t))\right)
    =
    rR(t)R'(t)
    -g\frac{G(t)}{p(G(t))}G'(t)=0.
\]
The saturation threshold is
\[
    R^*=\sqrt{\frac{2g}{r}H(G)}.
\]
Whenever $x\ge R^*$, elimination of the green army
leaves
\[
    s(x)=\sqrt{x^2-(R^*)^2}
\]
red survivors. Thus, although the intelligence
function changes the threshold, the survival
function above the threshold has the same form
as in the classical Lanchester model.

For several strongholds and initial army size 
$R>\sum_iR_i^*$, Corollary~\ref{cor:proportional-allocation}
therefore gives the unique survivor-maximizing
allocation subject to eliminating all insurgents:
\[
    R_i=\frac{R_i^*}{\sum_{j=1}^n R_j^*}R.
\]
This recovers the proportional allocation obtained
in \cite{KKS}.

That paper also considers an extension in which
insurgents remaining in unattacked or unsuccessfully
attacked strongholds cause civilian casualties.
The resulting objective combines government losses
and civilian casualties; see \cite{KKS} for details.

\subsection{Implicitly defined loss functions}
\label{sec:implicit-survival}

Explicit outcome formulas are not necessary for
applying the allocation results. Consider the model
from \cite{KM}:
\[
    R'(t)=-gG(t)\bigl(\mu+(1-\mu)R(t)\bigr),
    \qquad
    G'(t)=-rR(t)\bigl(\nu+(1-\nu)G(t)\bigr),
\]
with $R(0)=x$, $G(0)=G$, and $\mu,\nu\in(0,1)$.
The parameters interpolate between aimed-fire
and encounter-dependent attrition terms.

Define
\[
    \Phi_\eta(u)=
    \frac{u}{1-\eta}
    -\frac{\eta}{(1-\eta)^2}
       \ln\left(1+\frac{1-\eta}{\eta}u\right),
    \qquad \eta\in(0,1).
\]
Then
\[
    \Phi_\eta'(u)=\frac{u}{\eta+(1-\eta)u},
    \qquad
    \Phi_\eta''(u)=
       \frac{\eta}{(\eta+(1-\eta)u)^2}.
\]
The conserved quantity is
\[
    r\Phi_\mu(R(t))-g\Phi_\nu(G(t)).
\]

The functions $\Phi_\mu,\Phi_\nu$ are continuous
and strictly increasing on $[0,\infty)$, vanish
at zero, and tend to infinity. Hence there is a
unique saturation threshold satisfying
\[
    \Phi_\mu(R^*)=\frac gr\Phi_\nu(G).
\]
For $0\le x\le R^*$, the number $G_S(x)$ of green
survivors is determined by
\begin{equation}\label{implicit-gs}
    \Phi_\nu(G_S(x))
    =\Phi_\nu(G)-\frac rg\Phi_\mu(x).
\end{equation}
For $x\ge R^*$, the red survival function is
similarly determined by
\[
    \Phi_\mu(s(x))
    =\Phi_\mu(x)-\frac gr\Phi_\nu(G).
\]

To apply the enemy loss results, it suffices to
establish that $k(x)=G-G_S(x)$ is strictly increasing
and strictly convex on $[0,R^*]$.
For $0<x<R^*$, implicit differentiation gives
\[
    \Phi_\nu'(G_S)G_S'
    =-\frac rg\Phi_\mu'(x),
\]
and therefore
\[
    G_S'=-\frac rg
          \frac{\Phi_\mu'(x)}{\Phi_\nu'(G_S)}<0.
\]
Differentiating again yields
\[
    \Phi_\nu''(G_S)(G_S')^2
    +\Phi_\nu'(G_S)G_S''
    =-\frac rg\Phi_\mu''(x),
\]
so
\[
    G_S''=
    -\frac{
        \frac rg\Phi_\mu''(x)
        +\Phi_\nu''(G_S)(G_S')^2
    }{\Phi_\nu'(G_S)}
    <0.
\]
The implicit formula also gives continuity at both
endpoints. Thus $G_S$ is strictly decreasing and
strictly concave, and $k$ is strictly increasing
and strictly convex on $[0,R^*]$.

To establish NP-hardness, fix $\mu,\nu\in(0,1)$
and, for each positive integer $q$, choose
\[
    G=q,\qquad g=1,\qquad
    r=\frac{\Phi_\nu(q)}{\Phi_\mu(q)}.
\]
Here $\mu$ and $\nu$ are fixed, and each normalized
battlefield instance is specified by the positive integer~$q$;
the parameters above are determined by this index.
Then $R^*=q$, and the corresponding normalized
loss function $\ell_q$ is uniquely determined by
\[
    \Phi_\nu(q-\ell_q(x))
    =
    \Phi_\nu(q)
    \left(1-\frac{\Phi_\mu(x)}{\Phi_\mu(q)}\right),
    \qquad 0\le x\le q.
\]
The preceding calculations show that $\ell_q$ is
continuous, strictly increasing, and strictly
convex, with
\[
    \ell_q(0)=0,\qquad \ell_q(q)=q.
\]
Hence Theorem~\ref{np-hard-problem} applies to this
prescribed family.

For arbitrary positive battlefield parameters $g, r$,
the loss functions retain these properties.
Consequently, if $R,R_1^*,\dots,R_n^*$ are integers,
Theorem~\ref{pseudopolynomial-algo} gives an exact
algorithm using $O(n^2R)$ operations and $O(R)$
auxiliary space. 

This example illustrates the separation between
combat dynamics and allocation: implicit relations
can establish the required properties of the
outcome functions even when no explicit formulas
for those functions are available.

%

\end{document}